%% file: Torus_NoC4_VxArbo.tex
\documentclass[12pt]{article}
\usepackage{fullpage,amsthm,amssymb,amsmath}
\usepackage{graphicx,algorithmic,algorithm,amsfonts, verbatim}
\usepackage{enumerate}
\usepackage{tikz}

\newtheorem{theorem}{Theorem}[section]
\newtheorem{lemma}[theorem]{Lemma}
\newtheorem{claim}[theorem]{Claim}
\newtheorem{question}[theorem]{Question}

\theoremstyle{definition}

\begin{document}
\title{Vertex Arboricity of Toroidal Graphs with a Forbidden Cycle}
\author{Ilkyoo Choi\thanks{Department of Mathematics, University of Illinois at Urbana-Champaign. \texttt{ichoi4@illinois.edu}} \and Haihui Zhang\thanks{School of Mathematical Science, Huaiyin Normal University, Jiangsu, 223300, P. R. China. Research was supported by the NSFC Tianyuan foundation (Grant No. 11226285). This work was done while the author was visiting University of Illinois at Urbana-Champaign.}}
\date\today
\maketitle
\begin{abstract}
The vertex arboricity $a(G)$ of a graph $G$ is the minimum $k$ such that $V(G)$ can be partitioned into $k$ sets where each set induces a forest. 
For a planar graph $G$, it is known that $a(G)\leq 3$.
In two recent papers, it was proved that planar graphs without $k$-cycles for some $k\in\{3, 4, 5, 6, 7\}$ have vertex arboricity at most $2$. 
For a toroidal graph $G$, it is known that $a(G)\leq 4$. 
Let us consider the following question: do toroidal graphs without $k$-cycles have vertex arboricity at most $2$?
It was known that the question is true for $k=3$, and recently, Zhang proved the question is true for $k=5$. 
Since a complete graph on $5$ vertices is a toroidal graph without any $k$-cycles for $k\geq 6$ and has vertex arboricity at least three,  the only unknown case was $k=4$. 
We solve this case in the affirmative; namely, we show that toroidal graphs without $4$-cycles have vertex arboricity at most $2$. 
\end{abstract}

\section{Introduction}

Let $[n]=\{1, \ldots, n\}$. 
Only finite, simple graphs are considered. 
Given a graph $G$, let $V(G)$ and $E(G)$ denote the vertex set and edge set of $G$, respectively. 
The {\it vertex arboricity} of a graph $G$, denoted $a(G)$, is the minimum $k$ such that $V(G)$ can be partitioned into $k$ sets $V_1, \ldots, V_k$ where $G[V_i]$ is a forest for each $i\in[k]$. 
This can be viewed as a vertex coloring $f$ with $k$ colors where each color class $V_i$ induces a forest; namely, $G[f^{-1}(i)]$ is an acyclic graph for each $i\in [k]$. 
The {\it girth} of a graph $G$ is the length of the smallest cycle in $G$. 
Note that a graph with no cycles is a forest, and it has vertex arboricity 1. 

Vertex arboricity, also known as point arboricity, was first introduced by Chartrand, Kronk, and Wall \cite{1968ChKrWa} in 1968. 
Among other things, they proved Theorem~\ref{allplanar}. 
Shortly after, Chartrand and Kronk \cite{1969ChKr} showed that Theorem~\ref{allplanar} is sharp by constructing a planar graph with vertex arboricity $3$, and they also proved Theorem~\ref{outerplanar}.

\begin{theorem}\label{allplanar}\cite{1968ChKrWa}
If $G$ is a planar graph, then $a(G)\leq 3$. 
\end{theorem}
\begin{theorem}\label{outerplanar}\cite{1969ChKr}
If $G$ is an outerplanar graph, then $a(G)\leq 2$. 
\end{theorem}

We direct the readers to the work of Stein \cite{1971St} and Hakimi and Schmeichel \cite{1989HaSc} for a complete characterization of maximal plane graphs with vertex arboricity $2$. 

In 2008, Raspaud and Wang \cite{2008RaWa} not only determined the order of the smallest planar graph $G$ with $a(G)=3$, but also found several sufficient conditions for a planar graph to have vertex arboricity at most $2$ in terms of forbidden small structures; namely, they proved that a planar graph with either no triangles at distance less than $2$ or no $k$-cycles for some fixed $k\in\{3, 4, 5, 6\}$ has vertex arboricity at most $2$. 
Chen, Raspaud, and Wang \cite{2012ChRaWa} 
showed that forbidding intersecting triangles is also sufficient for planar graphs.
In \cite{2008RaWa}, Raspaud and Wang asked the following question:

\begin{question}\cite{2008RaWa}
What is the maximum integer $\mu$ where for all $k\in\{3, \ldots, \mu\}$, a planar graph $G$ with no $k$-cycles has $a(G)\leq 2$?
\end{question}

Raspaud and Wang's results imply $6\leq \mu\leq 21$. 
The lower bound was increased to $7$ by Huang, Shiu, and Wang \cite{2012HuShWa} since they proved planar graphs without $7$-cycles have vertex arboricity at most $2$. 

We completely answer the question for toroidal graphs, which are graphs that are embeddable on a torus with no crossings. 

Kronk \cite{1969Kr}  and Cook \cite{1974Co} investigated vertex arboricity on higher surfaces in 1969 and 1974, respectively. 

\begin{theorem}\label{higherSurfaces}\cite{1969Kr}
If $G$ is a graph embeddable on a surface of positive genus $g$, then $a(G)\leq \lfloor{9+\sqrt{1+48g}\over 4}\rfloor$. 
\end{theorem}

\begin{theorem}\label{higherSurfacesTri}\cite{1974Co}
If $G$ is a graph embeddable on a surface of genus $g$ with no $3$-cycles, then $a(G)\leq 2+\sqrt{g}$. 
\end{theorem}

\begin{theorem}\label{higherSurfacesGirth}\cite{1974Co}
If $G$ is a graph embeddable on a surface of positive genus $g$ with girth at least $5+4\log_3 g$, then $a(G)\leq 2$. 
\end{theorem}

Theorem~\ref{higherSurfaces} says every toroidal graph $G$ has $a(G)\leq 4$.
Theorem~\ref{higherSurfacesTri} says a toroidal graph with no $3$-cycles has vertex arboricity at most $3$, and Theorem~\ref{higherSurfacesGirth} only guarantees that toroidal graphs with girth at least $5$ have vertex arboricity at most $2$. 
Both of these cases were improved by Kronk and Mitchem \cite{1974KrMi} who showed Theorem~\ref{toroNo3}.
Recently, Zhang \cite{00Zh} showed Theorem~\ref{toroNo5}, which says that forbidding $5$-cycles in toroidal graphs is sufficient to guarantee vertex arboricity at most $2$. 

\begin{theorem}\label{toroNo3}\cite{1974KrMi}. 
If $G$ is a toroidal graph with no $3$-cycles, then $a(G)\leq 2$. 
\end{theorem}

\begin{theorem}\label{toroNo5}\cite{00Zh}. 
If $G$ is a toroidal graph with no $5$-cycles, then $a(G)\leq 2$. 
\end{theorem}

Since the complete graph on $5$ vertices is a toroidal graph with no cycles of length at least $6$ and has vertex arboricity $3$, the only remaining case is when $4$-cycles are forbidden in toroidal graphs; this is our main result. 

\begin{theorem}\label{result}
If $G$ is a toroidal graph with no $4$-cycles, then $a(G)\leq 2$. 
\end{theorem}


In section $2$, we will prove some structural lemmas needed in Section $3$, where we prove Theorem~\ref{result} using (simple) discharging rules. 
Note that Theorem~\ref{result} implies that every planar graph without $4$-cycles have vertex arboricity at most $2$, which is a result in \cite{2008RaWa}.


\section{Lemmas}

From now on, let $G$ be a counterexample to Theorem~\ref{result} with the fewest number of vertices. 
It is easy to see that $G$ must be $2$-connected and the minimum degree of a vertex of $G$ is at least $4$. 

A graph is $k$-regular if every vertex in the graph has degree $k$.
A set $S\subseteq V(G)$ of vertices is $k$-regular if every vertex in $S$ has degree $k$ in $G$.
A {\it triangular cycle} is a cycle adjacent to a triangle.
A (partial) 2-coloring $f$ of $G$ is {\it good} if each color class induces a forest. 

\begin{lemma}\label{partial}
If $V(G)$ contains a 4-regular set $S$ where $G[S]$ is a cycle $C$, then every good coloring $f$ of $G[V(G)\setminus S]$ that does not extend to all of $G$ has either 
\begin{enumerate}[\mbox{Case} 1:]
\item $f(v)$ the same for every vertex $v\not\in S$ that has a neighbor in $S$, or
\item $f(x)\neq f(y)$ for all $v\in S$ such that $N(v)\setminus S=\{x, y\}$ and $C$ is an odd cycle.
\end{enumerate}
\end{lemma}
\begin{proof}
Let $S=\{v_1, \ldots, v_s\}$ where $v_1, \ldots, v_s$ are the vertices of $C$ in this order. 
For each $i\in [s]$, let $\{x_i, y_i\}=N(v_i)\setminus S$. 
Obtain a good coloring $f$ of $G[V(G)\setminus S]$ by the minimality of $G$. 
We will show that if $f$ does not satisfy one of the two conditions in the statement, then $f$ can be extended to all of $G$. 

If $s$ is even and $\{f(x_i), f(y_i)\}=\{1, 2\}$ for each $i\in[s]$, then let $f(v_i)=
\begin{cases}
1 & \mbox{if $i$ is odd}\\
2 & \mbox{if $i$ is even}
\end{cases}$ 
to extend $f$ to all of $G$.

We know that there exists at least one index $j\in[s]$ where $f(x_j)= f(y_j)$ since we are not in Case 2. 
For each $i\in [s]$ where $f(x_i)=f(y_i)$, let $f(v_i)=
\begin{cases}
1 & \mbox{ if $f(x_i)=f(y_i)=2$ }\\
2 & \mbox{ if $f(x_i)=f(y_i)=1$}
\end{cases}$.
Now, consider the vertices of $C$ in cyclic order starting with $i=j$, and for $f(v_i)$ that is not defined yet, let $f(v_i)=
\begin{cases}
1 & \mbox{ if $f(v_{i-1})=2$ }\\
2 & \mbox{ if $f(v_{i-1})=1$}
\end{cases}$ for all $i$. 
We claim that this coloring $f$ is now a good coloring of all of $G$, which is a contradiction.

Note that $f$ cannot have a monochromatic cycle that only uses vertices of $V(G)\setminus S$. 
Also, $f$ cannot have a monochromatic cycle where $x_i, v_i, y_i$ are consecutive vertices on this cycle since $f(x_i)=f(v_i)=f(y_i)$ never happens. 
Moreover, $f$ cannot have a monochromatic cycle where $v_i, v_{i+1}, x_i$ are consecutive vertices on this cycle since $f(v_i)=f(v_{i+1})$ implies that $f(x_{i+1})=f(y_{i+1})\neq f(v_{i+1})$. 
Thus, a monochromatic cycle in $f$ must be $C$ itself, which is possible only in Case 1. 
\end{proof}

\begin{lemma}\label{reducible}
$V(G)$ does not contain a 4-regular set $S$ where $G[S]$ is a triangular cycle. 
\end{lemma}
\begin{proof}
Let $S=\{v_1, \ldots, v_s, u\}$, so that $u, v_1, v_2$ are the vertices of a triangle and let $C=S\setminus\{u\}$.
Let $v_1, \ldots, v_s$ be the vertices of $C$ in this order. 
For $i\in[2]$, let $v'_i$ be the neighbor of $v_i$ that is not in $S$.
We will obtain a good coloring of all of $G$ to show that $S$ does not exist. 
Obtain a good coloring $f$ of $G[V(G)\setminus C]$ by the minimality of $G$.

Assume that the first case of Lemma~\ref{partial} happens and without loss of generality, assume $f(v)=1$ for every vertex $v\not\in C$ that has a neighbor in $C$. 
For $i\in[s]\setminus\{1\}$, let $f(v_i)=2$ and let $f(v_1)=1$. 
If $f$ is not a good coloring, then in the graph induced by $f^{-1}(1)$, there must exist a cycle  where $v'_1, v_1, u, z$ are consecutive vertices on the cycle for some $z\in N(u)\setminus\{v_1, v_2\}$. 
Now, alter $f$ by letting $f(u)=2$ to obtain a good coloring of all of $G$.

Assume that the second case of Lemma~\ref{partial} happens and without loss of generality, assume $f(v'_1)=f(v'_2)=1$ and $f(u)=2$. 
Note that $s$ must be odd. 
For $i\in [s]\setminus\{1\}$, let $f(v_i)=
\begin{cases}
1 & \mbox{if $i$ is odd}\\
2 & \mbox{if $i$ is even}
\end{cases}$, let $f(v_1)=2$, and change $f(u)$ from $2$ to $1$.
If $f$ is not a good coloring, then in the subgraph induced by $f^{-1}(1)$, there must exist a cycle where $u$ and two of its neighbors that are not $v_1, v_2$ are consecutive vertices on the cycle.
Now for $i\in[s]\setminus\{1\}$, alter $f$ by letting $f(v_i)=
\begin{cases}
2 & \mbox{if $i$ is odd }\\
1 & \mbox{if $i$ is even}
\end{cases}$ (but keep $f(u)=2$) to obtain a good coloring of all of $G$.
\end{proof}

%
%

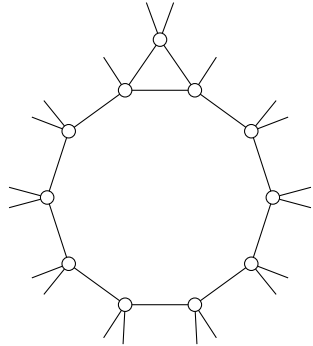
\begin{figure}[h]
	\begin{center}
		\input{reducible}
	\end{center}
  \caption{Forbidden Configuration. 
The white vertices do not have incident edges besides the ones drawn.}
  \label{fig:tikz:reducible}
\end{figure}

Let a vertex $v$ be {\it bad} if $d(v)=4$ and $v$ is incident to two triangles; 
a vertex is {\it good} if it is not bad. 
Let $H=H(G)$ be the graph where $V(H)$ is the set of triangles of $G$ incident to at least one bad vertex and let $uv\in E(H)$ if and only if there is a bad vertex of $G$ that is incident to both triangles that correspond to $u$, $v$.

\begin{claim}\label{components}
Each component of $H$ is either a cycle or a tree. 
\end{claim}
\begin{proof}
Assume for the sake of contradiction that $H$ has a component $D$ with a cycle $C$ where $C$ is not the entire component. 
Let $v\in V(D)\setminus V(C)$ be a vertex that has a neighbor in $V(C)$.
The graph in $G$ that corresponds to this structure is forbidden by Lemma~\ref{reducible}, which is a contradiction. 
See Figure~\ref{fig:tikz:helper}.
\end{proof}

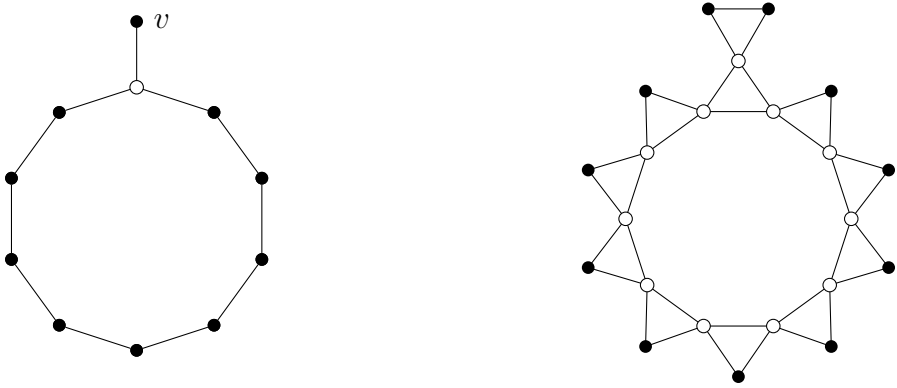
\begin{figure}[h]
	\begin{center}
		\input{helper}
	\end{center}
  \caption{The cycle $C$ in the proof of Claim~\ref{components} (left) and the corresponding graph in $G$ (right).
The white vertices do not have incident edges besides the ones drawn.
The black vertices may have other incident edges. }
  \label{fig:tikz:helper}
\end{figure}

Here is a lemma that will help later on. 

\begin{lemma}\label{tree-leaf}
Every $n$-vertex tree where with maximum degree $3$ has exactly $2$ more vertices of degree $1$ than vertices of degree $3$.
\end{lemma}
\begin{proof}
Let $z_i$ be the number of vertices of degree $i$. 
An $n$-vertex tree has $n-1$ edges and the sum of the degrees is twice the number of edges. 
Thus we have $n=z_1+z_2+z_3$ and $2(n-1)=z_1+2z_2+3z_3$. 
By eliminating $z_2$, we get $z_1=z_3+2$.
\end{proof}

\section{Discharging}

In this section, we will prove that $G$ cannot exist. 
Fix an embedding of $G$ and let $F(G)$ be the set of faces. 
We assign an \emph{initial charge} $\mu(z)$ to each $z\in V(G) \cup F(G)$, and then we will apply a discharging procedure to end up with {\it final charge} $\mu^*(z)$ at $z$.
We prove that the final charge has positive total sum, whereas the initial charge sum is at most zero.
The discharging process will preserve the total charge sum, and hence we find a contradiction to conclude that $G$ does not exist.

For every vertex $v\in V(G)$, let $\mu(v)=d(v)-6$, and for every face $f\in F(G)$, let $\mu(f)=2d(f)-6$. 
The total initial charge is zero since
\begin{align*}
\sum_{z\in V(G)\cup F(G)} \mu(z)
	=\sum_{v\in V(G)} (d(v)-6)+\sum_{f\in V(F)} (2d(f)-6) 
	=6|E(G)|-6|V(G)|-6|F(G)|
	\leq 0.
\end{align*}
The final equality holds by Euler's formula.

For the discharging procedure we introduce the notion of a {\it bank}, which serves as a placeholder for charges. 
For each component $D$ of the auxiliary graph $H(G)$, we will define a separate bank; let $b(D)$ denote the bank. 
We give each bank an initial charge of zero and we will show that either some vertex or some bank has positive final charge. 
The rest of this section will prove that the sum of the final charge after the discharging phase is positive. 

Recall that a vertex $v$ is {\it bad} if $d(v)=4$ and $v$ belongs to two triangles and a vertex is {\it good} if it is not bad. 
A good vertex $v$ is incident to a bank $b(D)$ if there is a vertex $u$ of $D$ where $v$ is incident to the triangle in $G$ that corresponds to $u$.
Note that each bad vertex of $G$ is an edge of $H(G)$. 

Here are the discharging rules:

\begin{enumerate}[(R1)]
\item Each face distributes its initial charge uniformly to each incident vertex. 

\item Each good vertex $v$ sends charge ${2\over 5}$ to each bank $b(D)$ each time $v$ is incident to $b(D)$. 

\item For each component $D$ of $H(G)$, the bank $b(D)$ sends charge $2\over 5$ to each bad vertex in $G$ that corresponds to an edge in $D$.

\end{enumerate}

It is trivial that each face has nonnegative final charge. 
Moreover, each face $f$ with $d(f)\geq 5$ sends charge ${\mu(f)\over d(f)}={2d(f)-6\over d(f)}\geq {4\over 5}$ to each incident vertex. 
We will first show that each vertex has nonnegative final charge.
Then we will show that either some bank or some vertex has positive final charge.

Note that since $G$ has no $4$-cycles, each vertex $v$ is incident to at most $\lfloor {d(v)\over 2}\rfloor$ triangles, and therefore at most $\lfloor{d(v)\over 2}\rfloor$ banks.

\begin{claim}\label{vertex}
Each vertex has nonnegative final charge.
Moreover, each vertex of degree at least $5$ has positive final charge.
\end{claim}
\begin{proof}
A vertex $v$ with $d(v)\geq 6$ has nonnegative initial charge and receives at least ${4\over 5}\cdot{d(v)\over 2}$ after (R1). 
Since $v$ is incident to at most $\lfloor{ d(v)\over 2}\rfloor$ incident banks, $\mu^*(v)\geq{4\over 5}\cdot{d(v)\over 2}-{2\over 5}\cdot\lfloor{d(v)\over 2}\rfloor>0$.
A vertex $v$ with $d(v)=5$ will receive charge from at least 3 incident faces and will give charge to at most 2 incident banks. 
Therefore, $\mu^*(v)\geq -1+3\cdot{4\over 5}-2\cdot{2\over 5}> 0$. 

A good vertex $v$ with $d(v)=4$ will receive charge from at least 3 faces and will give charge to at most 1 incident bank. 
Therefore, $\mu^*(v)\geq -2+3\cdot{4\over 5}-{2\over 5}= 0$. 
A bad vertex $v$ will receive charge at least ${4\over 5}$ from two faces and $2\over 5$ from exactly one bank. 
Therefore, $\mu^*(v)\geq -2+ 2\cdot{4\over 5}+{2\over 5}= 0$. 
\end{proof}

Given a component $D$ of $H(G)$, since an edge of $D$ corresponds to a bad vertex of $G$, we need to check that $b(D)$ has enough charge for each edge of $D$. 

\begin{claim}\label{cycle-bank}
Each bank $b(D)$ where $D$ is a cycle has nonnegative final charge. 
\end{claim}
\begin{proof}
Assume $D$ is a cycle $C$ with $n$ vertices. 
Since $D$ is a cycle, each triangle in $G$ that corresponds to a vertex in $D$ must be incident to one good vertex; each good vertex will send charge $2\over 5$ to $b(D)$. 
Thus, $b(D)$ receives charge ${2\over 5}n$ and there are $n$ edges in $D$ so $b(D)$ has nonnegative final charge.
\end{proof}

\begin{claim}\label{tree-bank}
Each bank $b(D)$ where $D$ is a tree has positive final charge. 
\end{claim}
\begin{proof}
Assume $T$ has $n$ vertices. 
$T$ has maximum degree at most 3 since a triangle in $G$ cannot be incident to more than 3 bad vertices.
For $i\in[3]$, let $z_i$ be the number of vertices of degree $i$ in $T$.

Each triangle in $G$ that corresponds to a degree 1 vertex in $T$ is incident to 2 good vertices, and each triangle in $G$ that corresponds to a degree 2 vertex in $T$ is incident to 1 good vertex. 
Thus $b(T)$ gets charge ${4\over 5}z_1+{2\over 5}z_2$, and must spend ${2\over 5}|E(T)|={2\over 5}(n-1)$. 
Since $z_1=z_3+2$ by Lemma~\ref{tree-leaf}, it follows that ${4\over 5}z_1+{2\over 5}z_2={2\over 5}n+{4\over 5}>{2\over 5}n-{2\over 5}$.
Thus, $b(T)$ has positive final charge.
\end{proof}

If $H(G)$ has a component that is a tree $T$, then $b(T)$ has positive final charge.
If $H(G)$ has a component that is a cycle, then there exists a vertex of degree at least 5 in $G$, and by Claim~\ref{vertex}, this vertex has positive final charge. 
If $H(G)$ has no components, then there are no bad vertices, and we are done since either some bank or some vertex will have positive final charge. 

\section{Acknowledgment}

The authors thank Alexandr V. Kostochka and Bernard Lidick\'y for improving the readability of the paper. 

\bibliographystyle{plain}
\bibliography{Torus_NoC4_VxArbo}

%
%

\end{document}

%% file: reducible.tex
\usetikzlibrary{patterns}
\usetikzlibrary{fadings}
\usetikzlibrary{shapes,snakes}

\def\e{1}
\def\es{0.6}
\def\R{2.1}
\def\r{1.5}

\def\confZ{
\draw
(0,0)  node[draw=none, label=above:] (c) {}

	\foreach \angle in {0, 36, ..., 324}{
		(\angle: \r) node[low, label=above:] (x\angle) {}
		(\angle:\r) -- (\angle+36:\r)
	}
	\foreach \angle in {0, 36, 144, 180, ..., 324}{
		(x\angle) -- ++ (\angle+15:\R/4)	
		(x\angle) -- ++ (\angle-15:\R/4)	
	}
	(x72) -- (72+18:\R) node[low](x){}
	-- (72+36:\r)
	(x) -- (72+18+4:\R+0.5)
	(x) -- (72+18-4:\R+0.5)
	(x72) -- ++ (72-15:\R/4)
	(x108) -- ++ (108+15:\R/4)
;
}

\begin{tikzpicture}
[
high/.style={inner sep=1.7pt, outer sep=0pt, circle, fill}, 
low/.style={inner sep=1.8pt, outer sep=0pt, circle, draw,fill=white}, 
]



\begin{scope}[xshift=0cm, yshift = 0cm] 
\confZ
\end{scope}


\end{tikzpicture}

%% file: helper.tex
\usetikzlibrary{patterns}
\usetikzlibrary{fadings}
\usetikzlibrary{shapes,snakes}

\def\e{1}
\def\es{0.6}
\def\R{2.1}
\def\r{1.5}

\def\confZ{
\draw
(0,0)  node[draw=none, label=above:] (c) {}

	\foreach \angle in {0, 36, ..., 324}{
		(\angle: \r) node[low, label=above:] () {}
		-- (\angle+18:\R)	node[high](x){}
		-- (\angle+36:\r)
		(\angle:\r) -- (\angle+36:\r)
	}

	(90:\R) node[low](x){}
	(x) -- ++ (60:0.8) node[high] (a){}
	(x) -- ++ (120:0.8) node[high] (b){}
	(a)--(b)
;
}

\def\confH{
\draw
(0,0)  node[draw=none, label=above:] (c) {}

	\foreach \angle in {0, 36, ..., 324}{
		(\angle+18:\R/1.2)	node[high](x1){}
		(\angle+54:\R/1.2)	node[high](x2){}
		(x1)--(x2)
	}
	(90:\R/1.2) node[low](){}
	(90:\R/1.2) -- (90:\R/0.8) node[high, label=right:$v$]{}
;
}

\begin{tikzpicture}
[
high/.style={inner sep=1.7pt, outer sep=0pt, circle, fill}, 
low/.style={inner sep=1.8pt, outer sep=0pt, circle, draw,fill=white}, 
]



\begin{scope}[xshift=0cm, yshift = 0cm] 
\confH
\end{scope}

\begin{scope}[xshift=8cm, yshift = 0cm] 
\confZ
\end{scope}

\end{tikzpicture}